\documentclass[12pt,reqno]{amsart}

\usepackage{amsmath, amsthm, amsopn, amssymb, enumerate}

\usepackage{latexsym }
\usepackage{graphics}
\usepackage{graphicx}
\usepackage{verbatim} 
\usepackage{hyperref} 
\usepackage[mathscr]{eucal}

\newtheorem{thm}{Theorem}[section]
\newtheorem{lemma}[thm]{Lemma}

\newtheorem{prop}[thm]{Proposition}

\newtheorem{example}[thm]{Example}

\newtheorem*{claim*}{Claim}

\theoremstyle{definition}

\newtheorem{defn}[thm]{Definition}

\newcommand{\ds}{\displaystyle}

\def\C{\mathcal{C}}

\def\M{\mathcal{M}}

\def\Q{\mathcal{Q}}

\def\S{\mathcal{S}}

\def\N{\mathbb{N}}

\def\vx{\mathbf{x}}
\def\vy{\mathbf{y}}
\def\vz{\mathbf{z}}

\def\le{\leqslant}
\def\ge{\geqslant}

\def\eps{\varepsilon}
\def\<{\langle}
\def\>{\rangle}

\date{\today}

\author{Gonzalo Fiz Pontiveros}

\author{Simon Griffiths}

\author{Robert Morris} 

\author{David Saxton}

\author{Jozef Skokan}

 \address{
   Gonzalo Fiz Pontiveros, Simon Griffiths, Robert Morris, David Saxton \hfill\break
    IMPA, Estrada Dona Castorina 110, Jardim Bot\^anico, Rio de Janeiro, RJ, Brasil
 }
 \email{\{gf232|sgriff|rob|saxton\}@impa.br}
 
  \address{
   Jozef Skokan \hfill\break
    Department of Mathematics, LSE, Houghton Street, London, WC2A 2AE, England
 }
 \email{j.skokan@lse.ac.uk}

\thanks{Research supported in part by: CNPq bolsas PDJ (GFP, SG, DS), a CNPq bolsa de Produtividade em Pesquisa (RM). This work was carried out during a visit of JS to IMPA in November 2012}

\begin{document}

\title{On the Ramsey number of the triangle and the cube}

\begin{abstract}
The Ramsey number $r(K_3,Q_n)$ is the smallest integer $N$ such that every red-blue colouring of the edges of the complete graph $K_N$ contains either a red $n$-dimensional hypercube, or a blue triangle. Almost thirty years ago, Burr and Erd\H{o}s conjectured that $r(K_3,Q_n) = 2^{n+1} - 1$ for every $n \in \N$, but the first non-trivial upper bound was obtained only recently, by Conlon, Fox, Lee and Sudakov, who proved that $r(K_3,Q_n) \le 7000 \cdot 2^n$. Here we show that $r(K_3,Q_n) = \big( 1 + o(1) \big) 2^{n+1}$ as $n \to \infty$.
\end{abstract}

\maketitle


\pagestyle{myheadings}
\markboth{}{}
\thispagestyle{empty}
\section{Introduction}
 
In 1983, Burr and Erd\H{o}s~\cite{BE} began the systematic study of the Ramsey numbers of small cliques and large sparse graphs; that is, the study of which graphs must occur in an $n$-vertex graph with no independent set of (constant) size $s$.  Their paper contained many conjectures and open problems, all but one of which have now been resolved (see~\cite{NR}). In this paper we take an important step towards resolving the single outstanding open question, by determining asymptotically the Ramsey number of the triangle and the hypercube. 
  
It follows from the classical theorem of Ramsey~\cite{Ramsey} that for any graphs $G$ and $H$, there exists an integer $N$ such that every red-blue colouring of the edges of the complete graph $K_N$ contains either a blue copy of $G$, or a red copy of $H$. The minimum such $N$ is called the \emph{Ramsey number of $G$ and $H$}, and is denoted $r(G,H)$. The problem of determining Ramsey numbers is among the most extensively-studied and notoriously difficult in Combinatorics. 

One special case of this general problem in which substantial progress has been made is when $G$ is a clique of some fixed size $s$, and $H$ is a large, fairly sparse graph. One of the first significant results of this type was obtained by Chv\'atal~\cite{Ch}, who proved that $r(K_s,T) = (s - 1)(n - 1) + 1$ for every $s \in \N$ and every tree $T$ on $n$ vertices. The lower bound in this statement is easy to see, and in fact holds in much greater generality: indeed, a collection of $s-1$ disjoint red cliques, each of size $n-1$, contains no connected red graph on $n$ vertices, and no blue graph $G$ of chromatic number $s$, and hence
\begin{equation}\label{eq:def:sgood}
r(K_s,H) \, \ge \, (s - 1)\big( v(H) - 1 \big) + 1
\end{equation}
for any connected graph $H$. Motivated by this construction and by Chv\'atal's theorem, Burr and Erd\H{o}s~\cite{BE} conjectured that equality holds in~\eqref{eq:def:sgood} for every fixed $s$ and all sufficiently sparse connected graphs $H$; in particular, for all such $H$ with bounded average degree. This conjecture was disproved by Brandt~\cite{Br}, who showed that there exist bounded degree graphs~$H$ with $r(K_3,H) > c \cdot v(H)$ for $c$ arbitrarily large. Furthermore, he showed that the conjecture fails (even in the case $s = 3$) for every large $d$-regular graph with sufficient expansion properties, and hence that it fails for almost every $d$-regular graph. On the other hand, it was shown by Burr and Erd\H{o}s~\cite{BE} that we have equality in~\eqref{eq:def:sgood} for every graph $H$ of bounded bandwidth\footnote{The bandwidth of a graph $H$ is defined to be the minimum $\ell \in \N$ for which there exists an ordering $v_1,\ldots, v_n$ of the vertices of $H$ such that every edge $v_iv_j$ satisfies $|i-j| \le \ell$.} and by Allen, Brightwell and Skokan~\cite{ABS} that the conjecture holds for every bounded degree graph of bandwidth $o(n)$. It follows that we have a natural rough dichotomy in the case of bounded degree graphs: equality holds in~\eqref{eq:def:sgood} for those graphs which have poor expansion properties, and fails otherwise.

For graphs with unbounded degrees much less is known. The strongest results obtained so far are due to Nikiforov and Rousseau~\cite{NR}, who proved the conjecture of Burr and Erd\H{o}s for $O(1)$-degenerate graphs which may be disconnected into components of size $o(n)$ by removing $n^{1-\eps}$ vertices, for some $\eps > 0$. As was observed in~\cite{CFLS}, together with the `separator theorem' of Alon, Seymour and Thomas~\cite{AST} this implies that equality holds in~\eqref{eq:def:sgood} for every sufficiently large graph which avoids a given minor, and hence for every sufficiently large planar graph. We remark that in fact the main results of both~\cite{ABS} and~\cite{NR} are considerably more general than those stated above; in particular, in~\cite{ABS} one may replace $K_s$ by any bounded size graph\footnote{To be precise, for more general graphs $G$, one must replace the term $1$ in~\eqref{eq:def:sgood} by $\sigma(G)$, the size of the smallest colour class in any $s$-colouring of $G$.} of chromatic number $s$, and in~\cite{NR} one may replace it by a very large collection of overlapping copies of~$K_s$. 

In this paper we shall study a specific family of graphs which not only have vertices of unbounded degree, but have unbounded \emph{average} degree. In doing so, we shall improve a recent result from~\cite{CFLS}, and take a significant step towards resolving the only question from~\cite{BE} left open by the work of~\cite{NR}. Let $Q_n$ denote the $n$-dimensional hypercube, i.e., the graph with vertex set $\{0,1\}^n$ and edges between pairs of vertices which differ in exactly one coordinate. This important family of graphs appears naturally in many different contexts, and its properties have been extensively-studied, including those relating to Ramsey Theory. For example, it is a long-standing conjecture of Burr and Erd\H{o}s that $r(Q_n,Q_n) = O(2^n)$, but the best known bounds (see~\cite{Con} and~\cite{FS}) are roughly the square of this function. 

We shall consider the problem, first proposed in~\cite{BE}, of determining the Ramsey numbers of cliques and hypercubes. It is straightforward to show (via a greedy embedding) that $r(K_s,Q_n) = O\big( n^{s-2} \cdot 2^n \big)$ for every fixed $s \in \N$, and this was essentially the best known upper bound until recently, when Conlon, Fox, Lee and Sudakov~\cite{CFLS} proved that
$$r(K_s,Q_n) \, \le \, C(s) \cdot 2^n$$
and hence determined $r(K_s,Q_n)$ up to a constant factor. In the concluding remarks of~\cite{CFLS}, the authors make the following comment: ``It would be of great interest to know whether the methods of this paper could be improved to give an approximate result of the form $r(K_s,Q_n) = \big( s-1+o(1) \big) \cdot 2^n$, even in the case of triangles. Such a result would likely be a necessary first step in resolving the original question of Burr and Erd\H{o}s." The following theorem provides this first step. 

\begin{thm}\label{Ramsey}
\[
 r(K_3,Q_n) \, = \, \big( 1 + o(1) \big) 2^{n+1}
\]
as $n \to \infty$.
\end{thm}

Our techniques do not appear to generalize easily to the case of larger cliques. Nevertheless, it seems likely that many of the ideas introduced below can be used in this more general context, and we plan to return to this topic in a future work. 

The strategy we shall use in order to prove Theorem~\ref{Ramsey} is roughly as follows. Given a two-colouring of $E(K_N)$ which contains no blue triangles, we will show (in Section~\ref{sec:decomposition}) that it can be split into two parts: a part with very few blue edges, and a part which consists of red \emph{$(m,s)$-snakes} (collections of $m$-cliques connected by copies of $K_{s,s}$) with very few blue edges between them. The larger of these two parts has at least $(1 + \gamma)2^n$ vertices; we shall show (in Sections~\ref{sec:dense_embed} and~\ref{sec:snake_embed}) how to find a red copy of $Q_n$ in that part. 

Indeed, in Section~\ref{sec:dense_embed} we shall adapt the technique introduced in~\cite{CFLS} in order to show that if the density of blue edges is sufficiently low (at most $1 / \log \log \log n$, say), then we can efficiently embed the hypercube $Q_n$ in $G_R$. Roughly speaking, the idea is to find a collection of disjoint red cliques, each associated with a subcube of $Q_n$, using the low blue density (and some simple double-counting) to ensure that there are few blue edges between cliques corresponding to adjacent subcubes. We shall then be able to greedily embed each subcube into its associated clique. Once we can no longer find any more red cliques, it will follow immediately that every vertex sends only $o(2^n/n)$ blue edges into the remaining vertices, and so we can complete the embedding greedily.  

On the other hand, we shall show in Section~\ref{sec:snake_embed} that any subset of a hypercube can be embedded in an $(m,s)$-snake, as long as $m$ and $s$ are sufficiently large. This embedding, which was inspired by an argument used in~\cite{ABS}, uses the fact that the bandwidth of $Q_n$ is $O(2^n / \sqrt{n})$. In fact, we shall need a slightly more technical statement (see Proposition~\ref{prop:snake}) which will allow us to avoid the blue neighbourhoods of vertices already embedded in other snakes, but the proof of this statement is not much more difficult.

Finally, combining the results of Sections~\ref{sec:dense_embed},~\ref{sec:snake_embed} and~\ref{sec:decomposition}, it is straightforward to prove Theorem~\ref{Ramsey}; we do so in Section~\ref{proofSec}.

\medskip
\noindent \textbf{Notation:} If~$G$ is a two-coloured complete graph, then we take the colours to be red and blue, and write $G_R$ and $G_B$ for the graphs formed by the red and blue edge sets respectively. We also write $N_B(u)$, $d_B(u)$ and $e_B(X,Y)$ for the neighbourhood and degree of a vertex in $G_B$, and the number of blue edges with one endpoint in $X$ and the other in $Y$, and similarly for $G_R$.  Throughout the paper, $\log$ denotes $\log_2$, and $\log_{(k)}$ denotes the $k^{th}$ iterated logarithm, so $\log_{(1)}(n)= \log(n)$ and $\log_{(k+1)}(n) = \log\big( \log_{(k)} (n) \big)$. We shall also omit irrelevant ceiling and floor symbols, and trust that this will cause the reader no confusion.

\section{An embedding lemma for dense red colourings}
\label{sec:dense_embed}

In this section we shall adapt the method of Conlon, Fox, Lee and Sudakov~\cite{CFLS} to prove the following proposition. 

\begin{prop}\label{prop:dense_embed}
Given any $\gamma > 0$ and $k \in \N$, there exists $n_0=n_0(\gamma,k)$ such that the following holds whenever $n\ge n_0$. If $H$ is a two-coloured complete graph on $(1 + \gamma)2^n$ vertices with no blue triangles and
\begin{equation}\label{eq:densebluedegrees}
d_B(u) \, \le \, \frac{2^n}{\log_{(k)} (n)}
\end{equation}
for every $u \in V(H)$, then $Q_n \subset H_R$.
\end{prop}

To save repetition, let us fix (for the rest of this section) a two-coloured complete graph $H$ with at least $(1 + 3\gamma)2^n$ vertices\footnote{Since $\gamma > 0$ is arbitrary, it is clearly sufficient to consider such an $H$.} and no blue triangles, where $0 < \gamma < 1/3$ and $k \in \N$ are fixed, and $n \in \N$ is sufficiently large. Let us assume also that $H$ satisfies~\eqref{eq:densebluedegrees}. 

We begin by introducing some notation. For each $d \ge 0$ and $\vx=(x_1,\dots, x_d) \in \{0,1\}^d$, let 
\[
Q_{\vx} \, = \, \big\{ (y_1,y_2,\dots,y_n) \in \{0,1\}^n \,:\, y_i = x_i \mbox{ for each } 1 \le i \le d \big\}
\]
denote the subcube of $Q_n$ consisting of points whose \emph{initial} coordinates\footnote{We shall call such a subcube an \emph{initial subcube} of $Q_n$.} agree with $\vx$. We call $d$ the \emph{co-dimension} of $Q_{\vx}$. Notice that  if $Q_{\vx}$ has co-dimension $d$, then every vertex $v\in V(Q_{\vx})$ has exactly $d$ neighbours in $V(Q_n) \setminus V(Q_{\vx})$.

We say that disjoint subcubes $Q$ and $Q'$ are \emph{adjacent} if there exist points $v\in V(Q)$ and $v'\in V(Q')$ that are adjacent in $Q_n$. Given two vectors $\vx \in \{0,1\}^d$ and $\vz \in \{0,1\}^{d'}$, 
we write 
$$d(\vx,\vz) \, = \, \sum_{i=1}^{\min\{d,d'\}} |x_i - z_i|$$ 
for the distance between $\vx$ and $\vz$. 
The advantage of using initial subcubes is, as observed in~\cite{CFLS}, that we can easily describe when two subcubes are disjoint or adjacent: the subcubes $Q_{\vx}$ and $Q_{\vz}$ are disjoint if and only if $d(\vx,\vz) >0$, and are adjacent if and only if~$d(\vx,\vz)= 1$. Let us also write $\vx \sim \vz$ if~$d(\vx,\vz)= 1$. 

\begin{defn}[Partial assignment of the cube]\label{def:partialass}
Let $m \in \N$, and suppose that we are given:
\begin{itemize}
\item[$(a)$] A sequence of integers, $0 \le d_1 \le \cdots \le d_m \le n$. \smallskip
\item[$(b)$] Disjoint sets $S_1, \ldots, S_m \subset V(H)$ of size $|S_i| = (1+\gamma)2^{n-d_i}$. \smallskip
\item[$(c)$] Vectors $\vx_i \in \{0,1\}^{d_i}$ for each $i \in [m]$, such that the subcubes $Q_{\vx_i}$ are disjoint. 
\end{itemize}
Then we say that $\big( d_i,S_i,\vx_i \big)_{i=1}^m$ is a {\em partial assignment of the cube} $Q_n$ into $H$ if each set $S_i$ induces a red clique in $H$, and moreover 
\begin{equation}\label{eq:def:partialass}
  |N_B(v) \cap S_i| \le \frac{\gamma}{d_i} \cdot 2^{n-d_i} \mbox{ for each pair $i<j$ with $\vx_i \sim \vx_j$ and every $v \in S_j$.}
\end{equation}
\end{defn}

Given a collection $\{Q_{\vx_1},\ldots,Q_{\vx_m}\}$ of subcubes of $Q_n$ as in the definition above, we shall write $\bigcup_{i=1}^m Q_{\vx_i}$ to denote the subgraph of $Q_n$ induced by the vertices $\bigcup_{i=1}^m V(Q_{\vx_i})$. The following lemma motivates Definition~\ref{def:partialass}.

\begin{lemma}\label{lem:partialass:embed}
Let $\big( d_i,S_i,\vx_i \big)_{i=1}^m$ be a partial assignment of the cube $Q_n$ into $H$. Then there exists an embedding
$$\varphi \,\colon \bigcup_{i=1}^m V(Q_{\vx_i}) \to V(H)$$
of $\bigcup_{i=1}^m Q_{\vx_i}$ into $H_R$, such that $\varphi\big(V(Q_{\vx_i}) \big) \subset S_i$ for each $i \in [m]$.
\end{lemma}

\begin{proof}
The strategy is simply to greedily embed each $Q_{\vx_i}$ into $S_i$ in turn, in the order $i = m,\ldots,1$. Since each set $S_i$ induces a clique in $H_R$, it is obvious that we can embed $Q_{\vx_m}$ into $S_m$. So suppose that we have successfully embedded $Q_{\vx_m}, \ldots, Q_{\vx_{i+1}}$ into $S_m, \ldots, S_{i+1}$. In order to embed $Q_{\vx_i}$ into $S_i$, we use the key observation that every vertex $v \in V(Q_{\vx_i})$ has $d_i$ neighbours in $V(Q_n) \setminus V(Q_{\vx_i})$, and thus $v$ has at most $d_i$ neighbours which are already embedded into $V(H)$. Moreover, each of these neighbours is embedded in a set $S_j$ with $j > i$ and $\vx_j \sim \vx_i$, and so, by~\eqref{eq:def:partialass}, each of these vertices has at most $\gamma 2^{n-d_i} / d_i$ blue neighbours in $S_i$. Since $S_i$ induces a red clique and 
\[
|S_i| \, \ge \, |Q_{\vx_i}| + \frac{\gamma}{d_i} \cdot 2^{n-d_i} \cdot d_i,
\]
it follows that we may embed the vertices of $Q_{\vx_i}$ into $S_i$ one by one, since there will be at least one available vertex at each step. Since $i \in [m-1]$ was arbitrary, this completes the proof of the lemma.
\end{proof}

The next lemma allows us to partition the vertices of $H$ into three sets: a partial assignment of $Q_n$ into $H$, a set which receives very low blue degree from every vertex of $G$, and a small set which we can discard. The reader should think of $a$ as $\log_{(j)} (n)$ for some $1 < j \le k$, and $b$ as roughly $\gamma \cdot 2^{a/2}$.

\begin{lemma}\label{lem:partialass:extend}
Let $\big( d_i,S_i,\vx_i \big)_{i=1}^m$ be a partial assignment of the cube $Q_n$ into $G$, let $a \ge 1$ and $d_m \le b \le n$ be integers, and let $A \subset V(H)$ be disjoint from $\bigcup_{i=1}^m S_i$. Suppose that $\bigcup_{i=1}^m Q_{\vx_i} \ne Q_n$, and that 
\begin{equation}\label{eq:bluedegintoA}
|N_B(v) \cap A| \le 2^{n-a}
\end{equation}
for every $v \in \bigcup_{i=1}^m S_i$. Then one of the following holds:
\begin{itemize}
 \item The partial assignment $\big( d_i,S_i,\vx_i \big)_{i=1}^m$ can be extended to $\big( d_i,S_i,\vx_i \big)_{i=1}^{m+1}$ with $d_{m+1}=b$ and $S_{m+1} \subset A$.
 \item There exists a set $C \subset A$ of size
$$|C| \, \ge \, |A| - \frac{b^2}{\gamma} \cdot 2^{n-a+1}$$
such that $|N_B(v) \cap C| < 2^{n-b+1}$ for every $v \in V(H)$.
\end{itemize}
\end{lemma}

\begin{proof}
We claim first that there exists $\vy \in \{0,1\}^{b}$ such that $Q_{\vy}$ is disjoint from $\bigcup_{i=1}^m Q_{\vx_i}$. Indeed, this follows since the set $U := V(Q_n) \setminus \bigcup_{i=1}^m V(Q_{\vx_i})$ is non-empty and $b \ge d_m \ge d_i$ for every $i \in [m]$, and thus $U$ can be partitioned into initial subcubes of co-dimension $b$. Now, for each $i \in [m]$ such that $\vx_i \sim \vy$, let us define the set of high-degree vertices into $S_i$ by
\[
 A(i) \, := \, \bigg\{ v \in A \,:\, |N_B(v) \cap S_i| \ge \frac{\gamma}{d_i} \cdot 2^{n-d_i} \bigg\},
\]
and set 
$$A' \, = \, \bigcup_{i \,:\, \vx_i \sim \vy} A(i).$$ 
Observe that 
$$\frac{\gamma}{d_i} \cdot 2^{n-d_i} \cdot |A(i)| \, \le \, e_B\big( A(i),S_i \big) \, \le \, 2^{n-a} |S_i| \, \le \, 2^{2n - a - d_i + 1},$$
where the first inequality follows from the definition of $A(i)$, the second follows from the condition~\eqref{eq:bluedegintoA}, and the third since $|S_i| \le 2^{n - d_i + 1}$. Furthermore $|\{i : \vx_i \sim \vy\}| \le b$, since $Q_{\vy}$ has co-dimension $b$, and so
\[
 |A'| \, \le \, \sum_{i \,:\, \vx_i \sim \vy} |A(i)| \, \le \, b \cdot \frac{d_m}{\gamma} \cdot 2^{n-a+1} \, \le \, \frac{b^2}{\gamma} \cdot 2^{n-a+1}.
\]
Set $C = A \setminus A'$. If $|N_B(v) \cap C| < 2^{n-b+1}$ for every $v \in V(H)$ then we are done, so assume that there exists a vertex $u \in V(H)$ with $|N_B(u) \cap C| \ge 2^{n-b+1}$. Since $G_B$ is triangle-free, it follows that $N_B(u) \cap C$ induces a clique in $H_R$, so let $S_{m+1}$ be an arbitrary subset of $N_B(u) \cap C$ of size $(1+\gamma)2^{n-b}$, and let $\vx_{m+1} = \vy$. By the definition of $A'$, it follows that 
$$|N_B(v) \cap S_i| < \frac{\gamma}{d_i} \cdot 2^{n-d_i}$$
for every $v \in S_{m+1} \subset C = A \setminus A'$ and every $i \in [m]$ with $\vx_i \sim \vx_{m+1}$. Hence $\big( d_i,S_i,\vx_i \big)_{i=1}^m$ can be extended to $\big(d_i, S_i, \vx_i \big)_{i=1}^{m+1}$ with $d_{m+1}=b$ and $S_{m+1} \subset A$, as required.
\end{proof}

Our strategy to find a copy of~$Q_n$ in $H_R$ is now straightforward:
\begin{itemize}
\item[$(i)$] Use Lemma~\ref{lem:partialass:extend} to find a partial assignment $\big(d_i, S_i,\vx_i \big)_{i=1}^m$ of the cube $Q_n$ into $H$, together with a large set $C \subset V(G)$ disjoint from $\bigcup_{i=1}^m S_i$, such that every vertex of $H$ has small blue degree into~$C$. \smallskip
\item[$(ii)$] Embed the graph $\bigcup_{i=1}^m Q_{\vx_i}$ into $\bigcup_{i=1}^m S_i$ using Lemma~\ref{lem:partialass:embed}. \smallskip
\item[$(iii)$] Extend this embedding, using the fact that the blue degrees into~$C$ are all small. 
\end{itemize}
Now we make these steps precise.

\begin{proof}[Proof of Proposition~\ref{prop:dense_embed}]
We fix the sequence $1 \le b_0 \le \ldots \le b_{k+1}$, where $b_j = 3 \log_{(k-j+2)}(n)$ for each $0 \le j \le k+1$, and assume that $n$ is large enough to have
$$d_B(u) \, \le \, \frac{2^n}{\log_{(k)}(n)} \, \le \, 2^{n-b_0}$$ 
for every $u \in V(H)$, 
$$\gamma^2 \cdot 2^{b_{j-1}} \,=\, \gamma^2 \cdot \big( \log_{(k-j+2)}(n) \big)^3 \,\ge\, 8 (k+1) \cdot b_{j}^2$$ 
for every $1 \le j \le k+1$, and 
$$\gamma \cdot 2^{b_{k+1}} \, = \, \gamma \cdot n^3 \, \ge \, n.$$ 
We begin by constructing a partial assignment $\big(d_i, S_i, \vx_i \big)_{i=1}^m$ and a set $C \subset V(H) \setminus \bigcup_{i=1}^m S_i$ as in Step~$(i)$ of the sketch above. This assignment will have $d_i \in \{ b_1+1, \ldots, b_{k+1}+1 \}$ for every~$i \in [m]$, and the blue degree condition will be
\[
 |N_B(v) \cap C| \le \frac{\gamma \cdot 2^n}{n} \qquad \mbox{ for every } v \in V(H).
\]
To obtain such a partial assignment, we repeatedly apply Lemma~\ref{lem:partialass:extend}. More precisely, we perform the following algorithm: 
\begin{itemize}
\item[0.] Set $j = 1$, $\ell = 0$ and $A = V(H)$. Repeat the following until STOP. \smallskip
\item[1.] If $\bigcup_{i=1}^\ell Q_{\vx_i} = Q_n$ or $j > k + 1$, then set $m = \ell$ and $C = A$, and STOP. \smallskip
\item[2.] Apply Lemma~\ref{lem:partialass:extend} to $\big(d_i, S_i, \vx_i \big)_{i=1}^{\ell}$ and $A$, with $a =  b_{j-1}$ and $b = b_j + 1$. \smallskip
  \begin{itemize}
  \item[$(a)$] If we obtain a partial assignment $\big(d_i, S_i, \vx_i \big)_{i=1}^{\ell+1}$ of the cube $Q_n$ into $H$,  with $d_{\ell+1} = b$ and $S_{\ell+1} \subset A$, then set $A := A \setminus S_{\ell+1}$ and $\ell := \ell + 1$, and repeat Step~1.\smallskip
  \item[$(b)$] Otherwise, we obtain a set $C_j \subset A$ with at least $|A| - (b_j^2 / \gamma) \cdot 2^{n-b_{j-1}+3}$ elements, with the property that $|N_B(v) \cap C_j| < 2^{n-b_j}$ for every $v \in V(H)$. In this case we set $A := C_j$ and $j := j+1$, and repeat Step 1. 
  \end{itemize}
\end{itemize}
We need to check that the conditions of Lemma~\ref{lem:partialass:extend} are always satisfied by $H$, with the parameters given in the algorithm. Indeed, we clearly have $a = b_{j-1} \ge 1$ and $d_\ell \le b_j + 1 \le n$ (since $1 \le b_0 \le \ldots \le b_{k+1}$), $\bigcup_{i=1}^\ell Q_{\vx_i} \ne Q_n$ by Step~1, and $A$ is disjoint from $\bigcup_{i=1}^\ell S_i$ by Step~3$(a)$. To see that~\eqref{eq:bluedegintoA} holds, observe that $A \subset C_{j-1}$ (where $C_0 = V(H)$), and so
$$|N_B(v) \cap A| \,\le\, |N_B(v) \cap C_{j-1}| \,\le\, 2^{n-b_{j-1}} \, = \, 2^{n-a}$$
for every $v \in V(H)$, by construction. 

By Lemma~\ref{lem:partialass:extend}, it follows that the algorithm runs as claimed. Hence $\big(d_i, S_i, \vx_i \big)_{i=1}^m$ is a partial assignment, and either $\bigcup_{i=1}^\ell Q_{\vx_i} = Q_n$ or $C \subset V(H) \setminus \bigcup_{i=1}^m S_i$ has the property that
$$|N_B(v) \cap C| \, \le \, 2^{n-b_{k+1}} \, \le \, \frac{\gamma \cdot 2^n}{n}$$
for every $v \in V(H)$. Moreover, we have
\begin{align}\label{eq:sizeofC}
|C| & \, \ge \, v(G) \, - \, \sum_{i=1}^m |S_i| \, - \, \sum_{j=1}^{k+1} \frac{b_j^2}{\gamma} \cdot 2^{n-b_{j-1}+3} \nonumber \\
& \, \ge \, (1+3\gamma)2^n  \, - \, (1+\gamma) \sum_{i=1}^m 2^{n-d_i} \, - \, \gamma 2^n \, = \, (1+\gamma) \bigg( 2^n - \sum_{i=1}^m 2^{n-d_i} \bigg) + \gamma 2^n,
\end{align}
since $8(k+1) \cdot b_{j}^2 \le \gamma^2 \cdot 2^{b_{j-1}}$ for each $1 \le j \le k+1$. 

Now, let us use the partial assignment $\big(d_i, S_i, \vx_i \big)_{i=1}^m$ and the set $C$ in order to embed the cube into~$H_R$, as in Steps $(ii)$ and $(iii)$ of the sketch. Indeed, by Lemma~\ref{lem:partialass:embed} there exists a partial embedding 
$$\varphi \,\colon \bigcup_{i=1}^m V(Q_{\vx_i}) \to V(H)$$
of $\bigcup_{i=1}^m Q_{x_i}$ into $H_R$, such that $\varphi\big(V(Q_{\vx_i}) \big) \subset S_i$ for each $i \in [m]$. We shall embed the remainder of $Q_n$ into $C$ greedily. Indeed, let
$$V(Q_n) \,\setminus\, \bigcup_{i=1}^m V(Q_{\vx_i}) \,=\, \big\{ q_1, \ldots, q_s \big\},$$ 
and suppose that we have embedded $q_1, \ldots, q_t$ into~$C$, and wish to embed $q_{t+1}$. The vertex $q_{t+1} \in V(Q_n)$ has at most~$n$ neighbours which are already embedded, each of which has at most $\gamma 2^n / n$ blue neighbours in $C$. Hence, by~\eqref{eq:sizeofC}, and since $t < s = 2^n - \sum_{i=1}^m 2^{n-d_i}$, it follows that there are at least 
$$|C| - t - \gamma 2^n \, > \, 0$$ 
choices for where to embed $q_{t+1}$, as required. This completes the embedding of the cube to $H_R$, and hence proves Proposition~\ref{prop:dense_embed}.
\end{proof}

\section{An embedding lemma into snakes}
\label{sec:snake_embed}

In this section we shall utilize the relatively low bandwidth of the cube $Q_n$, in order to embed it into (any member of) a family of graphs which arise naturally in the proof of Theorem~\ref{Ramsey} (see Section~\ref{sec:decomposition}), and which we term \emph{snakes}. We shall first define these graphs, and state our main result, and then provide some motivation. 

\begin{defn}
\label{def:snake}
Given a graph $G$, we say that a collection $\S = \{ M_1, \ldots, M_k \}$ of disjoint $m$-sets\footnote{We shall also think of $\S$ as a graph. Thus, abusing notation slightly, we shall write $V(\S) = \bigcup_{M \in \S} M$ for the vertex set of the $(m,s)$-snake $\S$.}  $M_j \subset V(G)$ is an \emph{$(m,s)$-snake} if the graph $H_\S(s)$ with vertex set $\S$ and edge set
$$E\big( H_\S(s) \big) \, = \, \bigg\{ \big\{ M,M' \big\} \in {\S \choose 2} \,:\, K_{s,s} \subset G[M,M'] \bigg\}$$
is connected, and $G[M]$ is a clique for every $M \in \S$.
\end{defn}

The aim of this section is to prove the following proposition.

\begin{prop}\label{prop:snake}
Let $n,m,s,\Delta \in \N$, and let $\S$ be an $(m,s)$-snake in a graph~$G$. Let $Q \subset Q_n$, and for each $x \in Q$, let $D_x \subset V(\S)$ be a ``forbidden'' set of size $|D_x| \le \Delta$. If
\begin{equation}\label{eq:prop:snake:conditions}
m \, \ge \, \frac{v(Q)}{|\S|} + s + \Delta \qquad \text{and} \qquad  s \, \ge \, 2\Delta + 8 \cdot |\S| \cdot \binom{n}{n/2},
\end{equation}
then there exists an embedding $\varphi \colon Q \to G\big[ V(\S) \big]$ such that $\varphi(x) \not\in D_x$ for every $x \in Q$.
\end{prop}

Before proving Proposition~\ref{prop:snake}, let us motivate the statement with a couple of simple examples. We shall write $(P_q)^t$ for the $t^{th}$ power of the path $P_q$, i.e., the graph with vertex set $[q]$ and edge set
\[
E\big( (P_q)^t \big) \, =\, \bigg\{ \big\{ i, j \big\} \in {[q] \choose 2} \,:\, \big| i - j \big| \le t \bigg\}.
\]

\begin{example}\label{ex:pw} 
Let $Q \subset Q_n$ be a subgraph of the cube, and let $t \ge 2\binom{n}{n/2}$. Then
$$Q\, \subset \, (P_q)^t,$$
where $q = v(Q)$. 
\end{example}

\begin{proof}
Choose an order $x_1 < \dots < x_q$ of the vertices of $Q$ such that the size of the corresponding subsets of $[n]$ is increasing. Note that if $x_i$ and $x_j$ are adjacent in $Q$ then there exists $k \in [n]$ such that $x_i$ belongs to level $k-1$ and $x_j$ to level $k$. By the choice of the ordering, it follows that
\[
| i - j | \, \le\, \binom{n}{k} + \binom{n}{k-1} \, \le \, 2\binom{n}{n/2} \, \le \, t,
\]
and hence $Q \subset (P_q)^t$, as required.
\end{proof}

The example above motivates the main idea of this section: that if we wish to embed a subset of $Q_n$ in a snake, we should first find a power of a path. Here is a slightly more complicated example.  

\begin{example} 
Let $s \ge 2\binom{n}{n/2}$, and suppose that the graph $G$ is composed of two (disjoint) cliques, each of size $2^{n-1}$, connected by a copy of $K_{s,s}$. Then $Q_n \subset G$.
\end{example}

\begin{proof}
By Example~\ref{ex:pw}, it will suffice to find a copy of $(P_{2^n})^s$ in $G$. Let $V(G) = A \cup B$, where $G[A]$ and $G[B]$ are both cliques and $|A| = |B| = 2^{n-1}$, and let $X \subset A$ and $Y \subset B$ be such that $G[X,Y] = K_{s,s}$. Then any ordering of the vertices of~$G$ consistent with the ordering
$$A \setminus X \, < \, X \, < \, Y \, < \, B \setminus Y$$
induces an embedding of $(P_{2^n})^s$ into $G$.
\end{proof}

We shall use a similar argument in order to find $Q \subset Q_n$ in a longer snake. The main complications are that the graph~$H_\S(s)$ may not resemble a path, and that we will need to embed the subgraph $Q \subset Q_n$ as we go along, in order to avoid the sets $D_x$. 

\begin{proof}[Proof of Proposition~\ref{prop:snake}]
Let $\S = \{ M_1, \ldots, M_k \}$, and recall that the graph $H_\S(s)$ is connected, since $\S$ is an $(m,s)$-snake. Consider an arbitrary spanning tree of $H_\S(s)$, and let $W = (w_0, \ldots, w_{2k})$ be a closed walk which traverses every edge of that spanning tree exactly twice. For each $0 \le j < 2k$, let $X_j \subset M_{w_j}$ and $Y_j \subset M_{w_{j+1}}$ be such that $G[X_j,Y_j] = K_{s,s}$, and for convenience define $X_{2k}=\emptyset$. We shall follow the walk $W$, using subsets of the $X_i$ and $Y_i$ at each step, and using the other vertices of $\S$ only when we arrive at a vertex of $H_\S(s)$ for the final time.

To be more precise, set $t = s/4k$ and define
$$T \, = \, \big\{ j \in [2k] \,:\, w_i \ne w_j \textup{ for all $i > j$} \big\}.$$
We embed $Q$ into $\S$ using the following algorithm:
\begin{itemize}
\item[0.] 
 \begin{itemize}
  \item[$(a)$] Choose an order $x_1 < \dots < x_q$ of the vertices of $Q$ such that the size of the corresponding subsets of $[n]$ is increasing, as in Example~\ref{ex:pw}.\smallskip
  \item[$(b)$] Set $j = \ell = 0$, and repeat the following steps until $\varphi(x_q)$ is chosen, or STOP. \smallskip
  \end{itemize}
\item[1.] 
  \begin{itemize}
  \item[$(a)$] Embed $x_{\ell+1},\ldots,x_{\ell+t}$ into $X_j$ one by one, subject to $\varphi(x_i) \not\in D_{x_i}$. \smallskip
  \item[$(b)$] Embed $x_{\ell+t+1},\ldots,x_{\ell+2t}$ into $Y_j$ one by one, subject to $\varphi(x_i) \not\in D_{x_i}$. \smallskip 
  \item[$(c)$] Set $j := j + 1$ and $\ell := \ell+2t$. \smallskip
  \end{itemize}
\item[2.] If $j \in T$ then repeat the following for as long as possible: \smallskip
  \begin{itemize}
  \item[$(a)$] Embed $x_{\ell+1}$ into $M_{w_j} \setminus \big( X_j \cup D_{x_{\ell+1}} \big)$, and set $\ell := \ell + 1$. \smallskip
  \end{itemize}
\item[3.] If $j \le 2k$, then return to Step  1. Otherwise, STOP. 
\end{itemize}
We claim that this procedure is always feasible, and that it gives an embedding of $Q$ into $G[V(\S)]$, i.e., all vertices of $Q$ are embedded and $\varphi$ is an injective homomorphism. To see that it is feasible, simply note that at most $2kt = s/2 < s - \Delta$ vertices of each set $X_j$ are used in the embedding, since at most $t$ are used at each stage. Similarly, at most $2kt$ vertices of $Y_j$ are used before stage $j$, since if $w_i = w_j$ for some $i < j$, then $i \not\in T$. 

To show that the map $\varphi$ given by the algorithm is an embedding of $Q$ into $G[V(\S)]$, we claim first that the algorithm does not terminate before all of the vertices of $Q$ are embedded. To see this, observe that otherwise $j = 2k+1$ at the end of the process, and that therefore we must have used all but at most $s+\Delta$ vertices of $M_j$ for every $j \in [k]$. It follows that
$$q \, > \, |\S| \cdot \big( m - s - \Delta \big),$$
which contradicts~\eqref{eq:prop:snake:conditions}. Finally, to see that $\varphi$ respects the edges of $Q$, simply note that, as before, if $x_a$ and $x_b$ are neighbours in $Q_n$, then $|a - b| \le 2{n \choose n/2} \le t$. Hence we have either $\varphi(x_a) \in X_j$ (for some $j \in [2k]$) and $\varphi(x_b) \in M_{w_j} \cup Y_j$, or $\varphi(x_a) \in Y_j$ and $\varphi(x_b) \in M_{w_{j+1}} \cup X_j$, or $\varphi(x_a) \in M_{w_j} \setminus \bigcup_{i=0}^{2k-1} \big( X_i \cup Y_i \big)$ and $\varphi(x_b) \in M_{w_j}$. It follows that we have indeed found an embedding of $Q$ into $G[V(\S)]$, and so this completes the proof of the proposition.
\end{proof}

\section{A structural decomposition of triangle-free colourings}
\label{sec:decomposition}

The aim of this section is to show that any two-coloured complete graph containing no blue triangles can be split into two parts: a dense part, with few blue edges, and a structured part (a collection of snakes with few blue edges between them). In order to slightly simplify the calculations below, let us set 
\begin{equation}\label{def:d}
d = \log \log \log n + 1.
\end{equation}
We shall prove the following proposition. 

\begin{prop}
\label{prop:decomposition}
Let $n \in \N$ be sufficiently large, and let $G$ be a two-coloured complete graph with no blue triangles, and with $2^n \le v(G) \le 2^{n+2}$. Then there exists a partition of $V(G)$ into sets $C \cup S_1 \cup \cdots \cup S_r$, for some $r \ge 0$, such that the following conditions hold:
\begin{itemize}
\item[$(a)$] $e\big( G_B[C] \big) \le \ds\frac{2^n |C|}{\log\log n}$
\end{itemize}
and, for every $i \in [r]$, there exists $n^{-1/3} \le s_i \cdot 2^{-n} \le n^{-1/4}$ such that
\begin{itemize}
\item[$(b)$] $G_R[S_i]$ contains a spanning $(m,s_i)$-snake $\S_i$, where $m = \Theta\bigg( \ds\frac{2^{n}}{\log \log n} \bigg)$.
\item[$(c)$] $|N_B(v) \cap S_i| \le \ds\frac{s_i}{\log \log n}$ for every $v \in S_{i+1} \cup \cdots \cup S_r$.
\end{itemize}
\end{prop}

Before proving the proposition, let's give a brief sketch of the proof. Throughout the process (after $j$ steps, say), we shall maintain a partition of $V(G)$ into sets
$$A_j \quad \textup{(for `active'),}  \qquad C = C_1 \cup \cdots \cup C_j \qquad \textup{and} \qquad S = S_1 \cup \cdots \cup S_j$$ 
satisfying, for each $1 \le i \le j$, that every vertex of $C_i$ sends few blue edges into $A_i \cup C_i$, and that each $G_R[S_i]$ contains a spanning $(m,s_i)$-snake. At each step of the process, we shall find sets $C_{j+1} \subset A_j$ and $S_{j+1} \subset A_j$ which maintain these properties. 

In order to do so, we consider a maximal collection $\M$ of disjoint red cliques in $A_j$ of a given size, $m = 2^{n-d} \approx 2^n / \log\log n$. Let $U$ denote the collection of vertices of these cliques, and note that every vertex of $G$ sends at most $m$ blue edges into $A_j \setminus U$. The key observation is that moreover, if any vertex of $G$ sends at least $s$ blue edges into two different cliques of $\M$, then these cliques are linked by a red copy of $K_{s,s}$, and hence lie in a red $(m,s)$-snake. 

It follows that if we partition $U$ into $(m,s)$-snakes, then every vertex can have high blue degree into at most one of these snakes. To finish the proof, we choose $s_{j+1}$ such that for every pair of cliques $K_1,K_2 \in \M$, the largest $s$ such that $K_1$ and $K_2$ are connected by a $K_{s,s}$ is either at least $s_{j+1}$, or at most $s_{j+1} / (\log\log n)^3$. (This is possible by the pigeonhole principle.) We can now choose an arbitrary $(m,s_{j+1})$-snake from our partition with vertex set~$S_{j+1}$, and move the vertices which send many blue edges to $S_{j+1}$ from $A_j$ into $C_{j+1}$. 

We now turn to the technical details of the proof sketched above. 

\begin{proof}[Proof of Proposition~\ref{prop:decomposition}]
Set $A_0 = V(G)$, and let $j \ge 0$. Suppose that we have found integers $s_1,\ldots,s_j$ and sets $A_1,\ldots,A_j$, $C_1,\ldots,C_j$ and $S_1,\ldots,S_j$ such that, for each $1 \le i \le j$, 
$$n^{-1/3} \le s_i \cdot 2^{-n} \le n^{-1/4}$$ 
and
$$V(G) = A_i \cup \big( C_1 \cup \cdots \cup C_i \big) \cup \big( S_1 \cup \cdots \cup S_i \big)$$
is a partition satisfying the following conditions:
\begin{itemize}
\item[$(i)$] $G_R[S_i]$ contains a spanning $(m,s_i)$-snake $\S_i$, where $m = 2^{n-d}$.\smallskip
\item[$(ii)$] $\big| N_B(v) \cap S_i \big| < 2^{-2d} s_i$ for every $v \in A_i$.\smallskip 
\item[$(iii)$] $\big| N_B(v) \cap (A_i \cup C_i) \big| \le 2^{n-d+1}$ for every $v \in C_i$. 
\end{itemize}
We will show how to find an integer $n^{-1/3} \le s_{j+1} \cdot 2^{-n} \le n^{-1/4}$ and sets $C_{j+1} \subset A_j$ and $S_{j+1} \subset A_j$, such that these conditions hold for $j+1$.\medskip

\noindent \textbf{Step 1: Choosing $s_{j+1}$.} Let $\M$ be a maximal collection of disjoint red cliques in $A_j$, each of size~$m = 2^{n-d}$, and suppose first that $|\M| > 0$. In order to choose~$s_{j+1}$, consider the weighted complete graph $K_\M$ on vertex set~$\M$, where the weight of the edge $\{M,M'\}$ is defined to be
\begin{equation}\label{def:weights}
w(M,M') \, = \, \max\Big\{ s \in \N \,:\, K_{s,s} \subset G_R[M,M'] \Big\}.
\end{equation}
That is, $w(M,M')$ is the largest integer~$s \in \N$ such that there exists a copy of $K_{s,s}$ in the red bipartite graph $G_R[M,M']$ induced by the vertex sets of the cliques $M$ and $M'$. We claim that there exists an integer $n^{-1/3} \le x \cdot 2^{-n} \le n^{-1/4}$ such that no edge weights of this graph lie in the interval 
$$I(x) = \bigg[ \frac{x}{2^{3d+2}}, \, x \bigg).$$
To see that such an $x$ exists, observe that $|\M| \le 2^{d+2}$, since $v(G) \le 2^{n+2}$, and therefore $e(K_\M) \le 2^{2d+3}$.  It follows that
\[
(2^{3d+2})^{e(K_\M)} \le \, 2^{(3d+2) 2^{2d+3}} \le \, 2^{(\log \log n)^3} \le \, n^{o(1)},
\]
and hence $x$ exists by the pigeonhole principle. Set $s_{j+1} = x$. 

\medskip

\noindent \textbf{Step 2: Choosing $S_{j+1}$ and $C_{j+1}$.} Consider the graph $H_\M(s_{j+1})$ on vertex set $\M$, with edge set
$$E\big( H_\M(s_{j+1}) \big) \, = \, \bigg\{ \big\{ M,M' \big\} \in {\M \choose 2} \,:\, w(M,M') \ge s_{j+1} \bigg\},$$
where $w(M,M')$ are the weights defined in~\eqref{def:weights}. Note that the connected components of $H_\M(s_{j+1})$ correspond to $(m,s_{j+1})$-snakes, and label these snakes (arbitrarily) as $\Q_1, \ldots, \Q_q$. Set $\S_{j+1} = \Q_1$ and let $S_{j+1}$ be the vertex set of $\S_{j+1}$. 

Now, in order to choose $C_{j+1}$, recall from the sketch that we wish this set to consist of the vertices of $A_j$ which send `many' blue edges into $S_{j+1}$. We therefore define
$$C_{j+1} \, = \, \bigg\{ v \in A_j \setminus S_{j+1} \,:\, \big| N_B(v) \cap M \big| \ge \frac{s_{j+1}}{2^{3d+2}} \mbox{ for some } M \in \S_{j+1} \bigg\}.$$

\medskip

\noindent \textbf{Step 3: Checking the conditions.} We next need to show that the integer $s_{j+1}$ and the partition  
$$V(G) = A_{j+1} \cup \big( C_1 \cup \ldots \cup C_{j+1} \big) \cup \big( S_1 \cup \ldots \cup S_{j+1} \big)$$
satisfy the conditions $(i)$, $(ii)$ and $(iii)$. Note that the conditions for $0 \le i \le j$ are unaffected by our choices at step $j+1$, so it will suffice to consider the case $i = j+1$. 

Condition~$(i)$ follows immediately from the construction, since $S_{j+1}$ is the vertex set of the red $(m,s_{j+1})$-snake $\S_{j+1}$. To see that Condition~$(ii)$ holds, note that $A_{j+1} \subset A_j \setminus C_{j+1}$, and that therefore $| N_B(v) \cap M | < s_{j+1} \cdot 2^{-(3d+2)}$ for every $M \in \S_{j+1}$ and every $v \in A_{j+1}$. It follows that
$$\big| N_B(v) \cap S_{j+1} \big| \, \le \, \sum_{M \in \S_{j+1}} \big| N_B(v) \cap M \big| \, < \, \frac{|S_{j+1}|}{2^{n-d}} \cdot \frac{s_{j+1}}{2^{3d+2}} \, \le \, \frac{s_{j+1}}{2^{2d}}$$ 
for every $v \in A_{j+1}$, as required.  

Finally, let us show that Condition~$(iii)$ holds. We claim that 
\begin{equation}\label{eq:conditionthreeintwoparts}
\bigg| N_B(v) \cap \bigcup_{i = 2}^q V(\Q_i) \bigg| \, \ll \, 2^{n-d} \qquad \textup{and} \qquad \bigg| N_B(v) \cap A_j \setminus \bigcup_{i = 1}^q V(\Q_i) \bigg| \, \le \, 2^{n-d}
\end{equation}
for every $v \in C_{j+1}$. The latter inequality is easy, since no vertex of $G$ sends more than $2^{n-d} - 1$ blue edges into the set $A_j \setminus \bigcup_{i = 1}^q V(\Q_i)$, by the maximality of $\M$. To prove the former bound, we claim  that if~$v \in C_{j+1}$, then 
\begin{equation}\label{eq:conditionthreefirstpart:claim}
|N_B(v)\cap M| < \frac{s_{j+1}}{2^{3d+2}} \quad \textup{for every $M \in \, \bigcup_{i = 2}^q \Q_i$}.
\end{equation}
Indeed, suppose that there exists an $M \in \, \bigcup_{i = 2}^q \Q_i$ with $|N_B(v)\cap M| \ge s_{j+1} \cdot 2^{-(3d+2)}$, and recall that, since $v \in C_{j+1}$, there exists~$M' \in \S_{j+1}$ such that $|N_B(v) \cap M' | \ge s_{j+1} \cdot 2^{-(3d+2)}$. Since~$G_B$ is triangle-free, it follows that every edge between~$N_B(v) \cap M$ and $N_B(v) \cap M'$ is red, so the weight of the edge~$\{M,M'\}$ in~$K_\M$ is at least~$s_{j+1} \cdot 2^{-(3d+2)}$. However, since there are no edges in~$K_\M$ with weight in the interval $I(s_{j+1})$, the weight must in fact be at least~$s_{j+1}$. This implies that~$M$ and~$M'$ are in the same connected component of~$H_\M(s_{j+1})$, a contradiction, and so we have proved~\eqref{eq:conditionthreefirstpart:claim}.

The first inequality in~\eqref{eq:conditionthreeintwoparts} now follows easily, since there are at most $v(G) / 2^{n-d} \le 2^{d+2}$ cliques in $\M$, and so 
$$\bigg| N_B(v) \cap \bigcup_{i = 2}^q V(\Q_i) \bigg| \, \le \, 2^{d+2} \cdot \frac{s_{j+1}}{2^{3d+2}} \, \le \, \frac{2^{n-2d}}{n^{1/4}} \, \ll \, 2^{n-d}$$
for every $v \in C_{j+1}$, as claimed. Combining the two inequalities in~\eqref{eq:conditionthreeintwoparts}, and noting that $A_{j+1} \cup C_{j+1} = A_j \setminus S_{j+1}$, we obtain Condition~$(iii)$.

\medskip

\noindent \textbf{Step 4: Completing the proof.} Since $|A_j|$ is strictly decreasing as long as $|\M| > 0$, we must eventually reach a set $A_j$ such that~$|\M| = 0$. When this happens, we set $C = C_1 \cup \cdots \cup C_j \cup A_j$ and $r = j$. Clearly~$C \cup S_1 \cup \cdots \cup S_r$ is a partition of $V(G)$, and so it only remains to check that the conditions~$(a)$, $(b)$ and~$(c)$ hold.  

Let us begin with Condition~$(c)$, which follows since 
$$|N_B(v) \cap S_i| \, \le \, 2^{-2d} s_i \, \le \, \ds\frac{s_i}{\log \log n}$$
for every $v \in S_{i+1} \cup \cdots \cup S_r \subset A_i$, by property~$(ii)$. Condition~$(b)$ is also easy to see, since each $\S_i$ is a spanning $(m,s_i)$-snake of $G_R[S_i]$ by property~$(i)$, and $m = 2^{n-d} = 2^{n-1} / \log\log n$. Finally, to prove that Condition~$(a)$ holds, we need to show that $e\big( G_B[C] \big) \le 2^{n-d+1}|C|$. To see this, simply note that there are no red cliques of size $2^{n-d}$ in $A_r$, by the maximality of~$\M$, and thus $|N_B(v) \cap A_r| < 2^{n-d}$ for every $v \in V(G)$. Hence, using property~$(iii)$,
\begin{align*}
 e\big( G_B[C] \big) & \, \le \, \sum_{i=1}^r \sum_{v \in C_i} \big| N_B(v) \cap (A_i \cup C_i) \big| + e\big( G_B[A_r] \big) \\
  & \, \le \, \sum_{i=1}^r  2^{n-d+1} |C_i| + 2^{n-d} |A_r| \, \le \, 2^{n-d+1} |C|
\end{align*}
as required. This completes the proof of Proposition~\ref{prop:decomposition}.
\end{proof}

\section{The proof of Theorem~\ref{Ramsey}}\label{proofSec}

Combining the results of the previous three sections, it is now easy to deduce Theorem~\ref{Ramsey}. In brief, we shall apply Proposition~\ref{prop:decomposition} to our two-coloured complete graph $G$, and obtain a partition $C \cup S_1 \cup \cdots \cup S_r$ of $V(G)$ such that $G_B[C]$ is sparse and each $S_j$ contains a red spanning $(m,s_j)$-snake. If $C$ contains more than half of the vertices of $G$, then we shall remove any high degree vertices in $G_B[C]$, and apply Proposition~\ref{prop:dense_embed}. If not, then we shall use Proposition~\ref{prop:snake} in order to find a copy of $Q_n$ in $G_R[S_1 \cup \cdots \cup S_r]$. The details follow.

\begin{proof}[Proof of Theorem~\ref{Ramsey}]
Let $\eps \in (0,1)$, and let $n \ge n_0(\eps)$ be sufficiently large. Let $G$ be a two-coloured complete graph on $(1+\eps)2^{n+1}$ vertices, and suppose that~$G_B$ is triangle-free. We claim that~$G_R$ contains a copy of~$Q_n$. To prove this, we apply Proposition~\ref{prop:decomposition} to the graph~$G$, to obtain a partition $V(G) = C \cup S_1 \cup \cdots \cup S_r$ and integers $s_1,\ldots,s_j$, with $n^{-1/3} \le s_i \cdot 2^{-n} \le n^{-1/4}$ for each $1 \le i \le j$, satisfying conditions $(a)$, $(b)$ and $(c)$ of the proposition. We shall find our copy of $Q_n$ in the larger of the sets $C$ and $S_1 \cup \cdots \cup S_r$.

\medskip

\noindent \textbf{Case 1: $|C| \ge v(G)/2$.} By Proposition~\ref{prop:decomposition}$(a)$, we have
$$e\big( G_B[C] \big) \le \ds\frac{2^n |C|}{\log\log n}.$$
We claim that in this case $Q_n \subset G_R[C]$. We will apply Proposition~\ref{prop:dense_embed}, but first we must remove the high degree vertices from $G_B[C]$. Indeed, set $\gamma = \eps/2$ and
$$C' \, = \, \bigg\{ v \in C \,:\, |N_B(v)\cap C| \ge \frac{2^n}{\log\log\log n} \bigg\},$$
and note that, counting edges, we have $|C'| \ll |C|$. Setting $H = G[C \setminus C']$, it follows that
$$v(H) \, \ge \, \big( 1 + \gamma \big) 2^n \qquad \textup{and} \qquad d_B(u) \, \le \, \frac{2^n}{\log_{(3)}(n)}$$
for every $u \in V(H)$. By Proposition~\ref{prop:dense_embed}, it follows that $Q_n \subset H_R$, as required.

\medskip

\noindent \textbf{Case 2: $|S_1 \cup \cdots \cup S_r| \ge v(G)/2$.} By Proposition~\ref{prop:decomposition}$(b)$ and~$(c)$, for each $i \in [r]$, the graph $G_R[S_i]$ contains a spanning $(m,s_i)$-snake $\S_i$, where $m = \Theta\big( 2^{n} / \log\log n \big)$, and
\begin{equation}\label{eq:bluedegreeintolatersnakes}
|N_B(v) \cap S_i| \le \frac{s_i}{\log \log n}
\end{equation}
for every $v \in S_{i+1} \cup \cdots \cup S_r$. Recall from~\eqref{def:d} the definition of $d$, and note that $|\S_j| = O(2^{d})$. To find an embedding of $Q_n$ into $G_R[S_1 \cup \cdots \cup S_r]$, we will split $Q_n$ into subcubes of size $2^{n-2d}$, and embed them into the snakes using Proposition~\ref{prop:snake}.

Indeed, let us choose an arbitrary assignment
$$\psi \colon \C \to \{\S_1,\ldots,\S_r\},$$
where $\C$ denotes the collection of subcubes of $Q_n$ of co-dimension $2d$, such that 
\begin{equation}\label{eq:assigningsubcubes}
\big| \psi^{-1}(\S_j) \big| \, \le \, \frac{|S_j|}{2^{n-2d}} - 1.
\end{equation}
Note that this is possible since
$$\sum_{j=1}^r \Big( |S_j| - 2^{n-2d+1} \Big) \, \ge \, \big( 1 + \eps \big) 2^n - 2^{n-d+3} \, \ge \, 2^n.$$ 
Let $Q(j)$ denote the subgraph of $Q_n$ induced by the union of the subcubes in $\psi^{-1}(\S_j)$. We embed $Q(j)$ into $\S_j$, for each $j \in [r]$, in reverse order. Indeed, suppose we have successfully defined the embedding $\varphi$ for $Q(r),\ldots,Q(j+1)$; we will show how to embed $Q(j)$ into $\S_j$. 

For each vertex $x \in Q(j)$, let $D_x \subset S_j$ denote the collection of vertices which are forbidden to $x$ by its already-embedded $Q_n$-neighbours, i.e., 
$$D_x \, = \, \bigg\{ y \in S_j \,:\, y \in N_B\big( \varphi(z) \big) \textup{ for some $z \in \bigcup_{i=j+1}^r Q(i)$ such that $z \sim x$ in $Q_n$} \bigg\}.$$
We claim that $|D_x| \ll s_j$ for every $x \in Q(j)$. Indeed, since $Q(j)$ is made up of subcubes of co-dimension $2d$, it follows that $x$ has at most $2d$ already-embedded neighbours. Moreover, by~\eqref{eq:bluedegreeintolatersnakes}, each of these has at most $s_j / \log \log n$ blue neighbours in $S_j$. Since $d \ll \log\log n$, it follows that $|D_x| \ll s_j$, as claimed.

Now, by~\eqref{eq:assigningsubcubes} we have 
$$v\big( Q(j) \big) \, \le \, |S_j| - 2^{n-2d} \, \le \, \big( m - 2s_j \big) |\S_j| \qquad \textup{and} \qquad s_j \gg |\S_j| {n \choose n/2}$$
since $n^{-1/3} \le s_j \cdot 2^{-n} \le n^{-1/4}$ and $|\S_j| = O(2^{d})$. Hence, by Proposition~\ref{prop:snake}, there exists an embedding $\varphi \colon Q(j) \to G[S_j]$ such that $\varphi(x) \not\in D_x$ for every $x \in Q(j)$. 

Since this holds for every $j \in [r]$, and by the definition of $D_x$, it follows that we can embed the entire cube $Q_n$ into $G_R[S_1 \cup \cdots \cup S_r]$. This completes the proof of Theorem~\ref{Ramsey}.
\end{proof}

\end{document}